%% file: Note__MTTOs_2.tex
\documentclass[twoside]{article}
\usepackage{amsmath,amssymb,amscd,amsthm,verbatim,alltt,amsfonts,array}
\usepackage{mathrsfs}
\usepackage[english]{babel}
\usepackage{latexsym}
\usepackage{amssymb}
\usepackage{euscript}
\usepackage{graphicx}
\input macros-bm.tex

\def\Z{\Bbb Z}

\def\C{\Bbb C}

\newtheorem{theorem}{Theorem}
\theoremstyle{plain}

\newtheorem{corollary}{Corollary}

\newtheorem{lemma}{Lemma}

\numberwithin{equation}{section}

\begin{document}

{\footnotesize% 
\hfill
\begin{tabular}{c}
Bull. Math. Soc. Sci. Math. Roumanie\\
Tome 6x (10x) No. x, 201x, xx--yy
\end{tabular}}

  \vskip 1.2 true cm
\setcounter{page}{1x}

\begin{center} {\bf A Note on Matrix Valued Truncated Toeplitz Operators  } \\ 
          {by}\\
{\sc Rewayat Khan}
\end{center}

\pagestyle{myheadings}
\markboth{A Note on Matrix Valued Truncated Toeplitz Operators}{Rewayat Khan}

\begin{abstract}
	We prove a spatial isomorphism between two spaces of matrix valued truncated Toeplitz operators.
\end{abstract}

\begin{quotation}

\end{quotation}

\thispagestyle{empty}

\section{Introduction}\label{Sec1}

Toeplitz operators are compressions of the multiplication operators on $L^{2}(\mathbb{T})$ to its closed subspace $H^{2}$. In the canonical basis of $H^{2}$, the matrix of a Toeplitz operator has constant entries on diagonals   parallel to the main diagonal.

In 2007, Donald Sarason has introduced the truncated Toeplitz operators (see \cite{Sa}), which have attracted so much attention in the last decade. Truncated Toeplitz operators are compressions of multiplication operators to certain subspaces of the Hardy-Hilbert space called   model spaces, which are invariant under the backward shift operator.

For $E$ a $d$-dimensional Hilbert space and a matrix valued inner function $\Theta$, it makes sense to consider $S^{*}$-invariant subspaces in $ H^2(E) $ These spaces, denoted $K_{\Theta}= H^{2}(E)\ominus\Theta H^{2}(E)$, are called vector valued model spaces.
In the paper \cite{KT}, we have developed the theory of matrix valued truncated Toeplitz operators, which are compressions of
multiplication operator on $K_{\Theta}$.

The purpose of this note is to prove a natural spatial isomorphism between two spaces of matrix valued truncated Toeplitz operators. This completes the results of~\cite{KT}.

\section{Preliminaries}

Let $\mathbb{D}$ denote the open unit disc and $\mathbb{T}$ the unit circle in the complex plane $\mathbb{C}$. Throughout the paper $E$ will denote a fixed Hilbert space of finite dimension $d$, and $\mathcal{L}(E)$ the algebra of bounded linear operators on $E$, which may be identified with $d\times d$ matrices. $\mathcal{L}(E)$ is a Hilbert space endowed with Hilbert-Schmidt norm.  $H^{2}(E)$ is the Hardy-Hilbert space of $E$-valued analytic functions on $\mathbb{D}$ whose coefficients are square summable; it is a closed subspace of $L^{2}(E)$, which is the space of all functions defined on the unit circle that have square summable convergent series.

The Poisson integral formula can be used to provide the
analytic extension of a function in $H^{2}(E)$. It
follows that $f(z)$ and $f(e^{it})$ determine each other.

 $L^{2}(\mathcal{L}(E))$ is the space of square summable Fourier series with coefficients in $\mathcal{L}(E)$.
The space $H^{2}(\mathcal{L}(E))$ is a closed subspace of $L^{2}(\mathcal{L}(E))$ whose Fourier coefficients corresponding to negative indices vanishes. The space of bounded operator valued function in $H^{2}(\mathcal{L}(E))$ is denoted by $H^{\infty}(\mathcal{L}(E)).$\par

%-------------------------------------------------------------------------------------------------------------------------------

The unilateral shift $S:H^{2}(E)\longrightarrow H^{2}(E)$ is defined by $Sf=zf$, and its adjoint $S^{*}$ (backward shift) is given by the formula;
$S^{*}f=\frac{f-f(0)}{z}.$

An inner function is an element $\Theta\in H^{\infty}(\mathcal{L}(E))$ whose boundary values are
almost everywhere unitary operators in $\mathcal{L}(E)$.

 To each non-constant inner function $\Theta$, there corresponds a model space $K_{\Theta}$ defined by
$$K_{\Theta}=H^{2}(E)\ominus \Theta H^{2}(E).$$
This terminology stems from the important role that $K_{\Theta}$ plays in the model theory of Hilbert space contractions.
Just like the Beurling-type subspace $\Theta H^{2}(E)$ constitute nontrivial invariant subspace for the unilateral shift $S$, the subspace $K_{\Theta}$ plays an analogous role for the backward shift $S^{*}$.

Let $P_{\Theta}$ denote the orthogonal projection of $H^{2}(E)$ on the $S^{*}$-invariant subspace $K_{\Theta}$. We then define the operator $S_{\Theta}$ on $K_{\Theta}$ by
$S_{\Theta}f=P_{\Theta}(zf)~~~ \forall f\in K_{\Theta}.$
Its adjoint is $S_{\Theta}^{*}=S^{*}|_{K_{\Theta}}.$
 If $\Theta\in H^{\infty}(\mathcal{L}(E))$ is an inner function, we define the new function by $\widetilde{\Theta}(z)=\Theta(\overline{z})^{*}$, then $\widetilde{\Theta}$ is also an inner function, and the corresponding model space is denoted by $K_{\widetilde{\Theta}}$.
 \begin{theorem} \cite{Fu}\label{eq: def of tau}
The operator $\tau: L^{2}(E)\longrightarrow L^{2}(E)$ defined by
\begin{equation}
(\tau f)(e^{it})=e^{-it}\Theta(e^{-it})^{*}f(e^{-it}),
\end{equation}
 is unitary and $\tau (K_{\Theta})=K_{\widetilde{\Theta}}$.
\end{theorem}

The operator $ \tau $ plays an important role in computations related to $K_{\Theta}$. For instance, it is used in the proof of the next result.

\begin{theorem}[see \cite{Fu}]
The operator $S_{\Theta}$ and $S_{\widetilde{\Theta}}^{*}$ are unitarily equivalent.
\end{theorem}
For further use we note that the adjoint of $\tau$ is given by
\begin{equation}
(\tau^{*}f)(e^{it})=e^{-it}\widetilde{\Theta}(e^{-it})^{*}f(e^{-it})
 \end{equation}
and satisfies $ \tau P_{\Theta}=P_{\widetilde{\Theta}}\tau$.

 \begin{lemma}\cite{KT}
 The subsets
 $$D=\bigg{\{(I-\Theta(z)\Theta(0)^{*})x: x\in E}\bigg\}, \quad D_{*}=\bigg{\{\frac{1}{z}(\Theta(z)-\Theta(0))x:x\in E}\bigg\},$$
are subspaces of $K_{\Theta}$ of dimension equal to the dimension of $E$.
 \end{lemma}
 \begin{corollary}
 The subsets
 $$\widetilde{D}=\bigg{\{(I-\widetilde{\Theta}(z)\Theta(0))x: x\in E}\bigg\}, \quad \widetilde{D}_{*}=\bigg{\{\frac{1}{z}(\widetilde{\Theta}(z)-\widetilde{\Theta}(0))x:x\in E}\bigg\},$$
are subspaces of $K_{\widetilde{\Theta}}$.
 \end{corollary}
  \begin{theorem}\cite{Fu}
 We have
 $$ (S_{\Theta}^{*}f)(z)=\left\{
                     \begin{array}{ll}
                       \frac{f(z)}{z} ~~~ ~~~~~~~~~~~~~~~~~~~~~~~~~~~~~~~~ for ~f\in D^{\perp}, \\
                       -\frac{1}{z}\big(\Theta(z)-\Theta(0)\big)\Theta(0)^{*}x ~~~~~~~~~ for ~f=\big(I-\Theta(z)\Theta(0)^{*}\big)x\in D;
                     \end{array}
                   \right.$$
$$ (S_{\Theta}f)(z)=\left\{
                     \begin{array}{ll}
                       zf(z) ~~~ ~~~~~~~~~~~~~~~~~~~~~~~~~~~~~for ~f\in D_{*}^{\perp}, \\
                       -\big(I-\Theta(z)\Theta(0)^{*}\big)\Theta(0)x ~~~~~~~~~ for ~f=\frac{1}{z}\big(\Theta(z)-\Theta(0)\big)x\in D_{*}.
                     \end{array}
                   \right.$$
 \end{theorem}
 For a fixed inner function $\Theta$ and $\Phi\in L^{2}(\mathcal{L}(E))$, the corresponding matrix valued truncated Toeplitz operator $A_{\Phi}:K_{\Theta}\longrightarrow K_{\Theta}$ is the densely defined operator on $K_{\Theta}^{\infty}=H^{\infty}\cap K_{\Theta}$ which acts by the formula
$$A_{\Phi}f=P_{\Theta}(\Phi f).$$
The function $\Phi$ here is called the symbol of $A_{\Phi}$.

The spaces of matrix valued truncated Toeplitz operators on $K_{\Theta}$ and $K_{\widetilde{\Theta}}$ are denoted by $\mathcal{T}_{\Theta}$ and  $\mathcal{T}_{\widetilde{\Theta}}$ respectively. For the inner functions $\Theta_{1}$ and $\Theta_{2}$, we say that $\mathcal{I}_{\Theta_{1}}$ and
 $\mathcal{T}_{\Theta_{2}}$ are spatially isomorphic if there is a unitary operator $\tau: K_{\Theta_{1}}\longrightarrow K_{\Theta_{2}}$ such that
$\tau \mathcal{T}_{\Theta_{1}}=\mathcal{T}_{\Theta_{2}}\tau$.\par

 In  \cite{KT} (see Theorem 4.2.1) a characterization of matrix valued truncated Toeplitz operators (MTTOs) by shift invariance is obtained. For a bounded operator on $K_{\Theta}$ we say that $A$ is shift invariant if
\begin{equation}
f,Sf\in K_{\Theta}~~~~~Q_{A}(f)=Q_{A}(Sf);
\end{equation}
where $Q_{A}$ is the associated quadratic form on $K_{\Theta}$, defined by $Q_{A}(f)=\langle Af,f\rangle$ for $f\in K_{\Theta}$.
\begin{theorem}\cite{KT}
	A bounded operator $A$ on $K_{\Theta}$ belongs to $\mathcal{T}_{\Theta}$ if and only if $A$ is shift invariant.
\end{theorem}\par

\section{ Main Result}

Our main result shows that the unitary $ \tau $ also implements a spatial isomorphism between the corresponding spaces of matrix valued truncated Toeplitz operators.

\begin{theorem}(Spatial Isomorphism Theorem)
$\mathcal{T}_{\Theta}$ and $\mathcal{T}_{\widetilde{\Theta}}$ are spatially isomorphic, that is there exist a unitary operator $\tau: K_{\Theta}\longrightarrow K_{\widetilde{\Theta}}$ (as defined in \ref{eq: def of tau} ) such that $\tau \mathcal{T}_{\Theta}=\mathcal{T}_{\widetilde{\Theta}} \tau$.
\end{theorem}
\begin{proof}
 Let $f\in \widetilde{D}^{\perp}_{*}$ then we have $\tau^{*}f\in D^{\perp}$.
It follows that $S_{\Theta}^{*}\tau^{*}f=\frac{\tau^{*}f}{z}\in D_{*}^{\perp}$ therefore $S_{\Theta}S_{\Theta}^{*}\tau^{*}f=\tau^{*}f$.
Now for $\Phi\in L^{2}(\mathcal{L}(E))$, consider $A_{\Phi}\in \mathcal{T}_{\Theta}$, then $\tau A_{\Phi}\tau^{*}\in \tau \mathcal{T}_{\Theta}\tau^{*}$ and using the fact $\tau S_{\Theta}=S_{\widetilde{\Theta}}^{*}\tau$ we have $\tau^{*}S_{\widetilde{\Theta}}=S_{\Theta}^{*}\tau^{*}$ and
\begin{align*}
Q_{\tau A_{\Phi}\tau^{*}}(Sf)
&=\langle \tau A_{\Phi}\tau^{*}S_{\widetilde{\Theta}}f,S_{\widetilde{\Theta}}f\rangle
=\langle A_{\Phi}\tau^{*}S_{\widetilde{\Theta}}f,\tau ^{*}S_{\widetilde{\Theta}}f \rangle\\
&=\langle A_{\Phi}S_{\Theta}^{*}\tau^{*}f,S_{\Theta}^{*}\tau^{*}f \rangle=\langle A_{\Phi}S_{\Theta}S_{\Theta}^{*}\tau^{*}f, S_{\Theta}S_{\Theta}^{*}\tau^{*}f \rangle\\
&=\langle A_{\Phi}\tau^{*}f, \tau^{*}f \rangle=\langle\tau A_{\Phi}\tau^{*}f, f \rangle\\
&=Q_{\tau A_{\Phi}\tau^{*}}(f).
\end{align*}
It follows that $\tau \mathcal{T}_{\Theta} \tau^{*}\subset \mathcal{T}_{\widetilde{\Theta}}$.

 For the other inclusion, let $f\in D^{\perp}_{*}$ then we have $S_{\Theta}f=zf$ and $zf\in \mathcal{D}^{\perp}$.
It follows that $S_{\Theta}^{*}S_{\Theta} f=f$. Now for $\Psi\in L^{2}(\mathcal{L}(E))$, consider $A_{\Psi}\in \mathcal{T}_{\widetilde{\Theta}}$, then $\tau^{*} A_{\Psi}\tau\in \tau^{*} \mathcal{T}_{\widetilde{\Theta}}\tau$ and using the fact $\tau S_{\Theta}=S_{\widetilde{\Theta}}^{*}\tau$ we have $S_{\widetilde{\Theta}}\tau=\tau S_{\Theta}^{*}$.
\begin{align*}
Q_{\tau^{*} A_{\Psi}\tau}(Sf)
&=\langle \tau^{*} A_{\Psi}\tau S_{\Theta}f,S_{\Theta}f\rangle\\
&=\langle A_{\Psi}\tau S_{\Theta}f,\tau S_{\Theta}f \rangle=\langle A_{\Psi}S_{\widetilde{\Theta}}\tau S_{\Theta}f,S_{\widetilde{\Theta}}\tau S_{\Theta}f \rangle\\
&=\langle A_{\Psi}\tau S_{\Theta}^{*} S_{\Theta}f,\tau S_{\Theta}^{*} S_{\Theta}f \rangle=\langle A_{\Psi}\tau f, \tau^{*}f \rangle\\
&=\langle\tau^{*} A_{\Psi}\tau f, f \rangle=Q_{\tau^{*} A_{\Psi}\tau}(f).
\end{align*}
It follows that $\tau^{*}\mathcal{T}_{\widetilde{\Theta}}\tau\subset \mathcal{T}_{\Theta}$. Therefore the equality follows from the two inclusions.
\end{proof}

\noindent\textbf{Acknowledgement}\textit{.}  The author is thankful to his family for continuous support. Furthermore the author is also thankful to Prof. Dan Timotin for their useful remarks.

\vskip 0,65 true cm

\medskip 

\smallskip\noindent
Received:  \\
Revised: \\
Accepted: 

\salt

\adresa{$^{(1)}$ Abdus Salam School of Mathematical Sciences. \\
E-mail: {\tt rewayat.khan@gmail.com}
}

%\adresa{$^{(2)}$ Address of the second author \\
%E-mail: {\tt xx\ap yy.zz}
%}

\end{document}

%% file: macros-bm.tex
\newcommand{\salt}{\vspace*{2.5mm}}     %Maia, 06.05.98

\def\boxit#1{\vbox{\hrule\hbox{\vrule\kern2pt
 \vbox{\kern4pt#1\kern2pt}\kern2pt\vrule}\hrule}}
\def\qed{\hfill
 $\hskip0.3cm
 \boxit{\hsize 2pt \vsize 10pt}$\bigskip\noindent}
\def\C{{\mathchoice {\setbox0=\hbox{$\displaystyle\rm C$}\hbox{\hbox
to0pt{\kern0.4\wd0\vrule height0.9\ht0\hss}\box0}}
{\setbox0=\hbox{$\textstyle\rm C$}\hbox{\hbox
to0pt{\kern0.4\wd0\vrule height0.9\ht0\hss}\box0}}
{\setbox0=\hbox{$\scriptstyle\rm C$}\hbox{\hbox
to0pt{\kern0.4\wd0\vrule height0.9\ht0\hss}\box0}}
{\setbox0=\hbox{$\scriptscriptstyle\rm C$}\hbox{\hbox
to0pt{\kern0.4\wd0\vrule height0.9\ht0\hss}\box0}}}}
\def\Q{{\mathchoice {\setbox0=\hbox{$\displaystyle\rm
Q$}\hbox{\raise
0.15\ht0\hbox to0pt{\kern0.4\wd0\vrule height0.8\ht0\hss}\box0}}
{\setbox0=\hbox{$\textstyle\rm Q$}\hbox{\raise
0.15\ht0\hbox to0pt{\kern0.4\wd0\vrule height0.8\ht0\hss}\box0}}
{\setbox0=\hbox{$\scriptstyle\rm Q$}\hbox{\raise
0.15\ht0\hbox to0pt{\kern0.4\wd0\vrule height0.7\ht0\hss}\box0}}
{\setbox0=\hbox{$\scriptscriptstyle\rm Q$}\hbox{\raise
0.15\ht0\hbox to0pt{\kern0.4\wd0\vrule height0.7\ht0\hss}\box0}}}}
\def\T{{\mathchoice {\setbox0=\hbox{$\displaystyle\rm
T$}\hbox{\hbox to0pt{\kern0.3\wd0\vrule height0.9\ht0\hss}\box0}}
{\setbox0=\hbox{$\textstyle\rm T$}\hbox{\hbox
to0pt{\kern0.3\wd0\vrule height0.9\ht0\hss}\box0}}
{\setbox0=\hbox{$\scriptstyle\rm T$}\hbox{\hbox
to0pt{\kern0.3\wd0\vrule height0.9\ht0\hss}\box0}}
{\setbox0=\hbox{$\scriptscriptstyle\rm T$}\hbox{\hbox
to0pt{\kern0.3\wd0\vrule height0.9\ht0\hss}\box0}}}}
\def\Z{{\mathchoice {\hbox{$\sf\textstyle Z\kern-0.4em Z$}}
{\hbox{$\sf\textstyle Z\kern-0.4em Z$}}
{\hbox{$\sf\scriptstyle Z\kern-0.3em Z$}}
{\hbox{$\sf\scriptscriptstyle Z\kern-0.2em Z$}}}}

\newcommand{\eqnoone}  %normala (fara section) - un numar
   {}
\newcommand{\eqnotwo}  %cu sectiuni - doua numere
   {}

\newcounter{alf}

\newcommand{\adresa}[1]{\par\vspace*{-11pt}
                        \begin{flushright}
                        {\small
                        #1}
                        \end{flushright}
                        }

%% file: Note__MTTOs_2.bbl
\begin{thebibliography}{999}                                                                                                %
\bibitem{BH} A. Brown, P.R. Halmos: {Algebraic properties of Toeplitz operators}, \emph{J. Reine
Angew}, Math. 213: 89102, 19631964.



   \bibitem{D}
      R.G. Douglas:                                                               %ReF
  {Banach algebra techniques in operator theory
  },volume 179 of Graduate Texts in Mathematics.
  Springer-Verlag, New York, 1998.





\bibitem{Fu} P. A. Fuhrmann: {On a class of finite dimensional
    contractive perturbations of restricted shift of finite
    multiplicity},  \emph{Israel J. Math.} 16 (1973), 162--175.



\bibitem{GMR}
      S.R. Garcia, J. Mashreghi, W.T. Ross:                                                               %ReF
  {Introduction to model spaces and their operators
  }. Cambridge: Cambridge University Press. 2016.

   \bibitem{KT}
{R. Khan, D. Timotin}:, Matrix valued truncated Toeplitz operators: Basic Properties. Complex Analysis and Oper. Theory, 2017,
DOI10.1007/s11785-017-0675-3.

\bibitem{Sa} D. Sarason: {Algebraic properties of truncated
    Toeplitz operators},  \emph{Oper. Matrices} 1 (2007),
    491--526.




\end{thebibliography}
